\documentclass[a4paper,11pt,reqno]{amsart}
\usepackage{amsmath}
\usepackage[T1]{fontenc}
\usepackage{amssymb}
\usepackage{amsthm,graphicx}
\usepackage{color}
\usepackage[lmargin=2.5 cm,rmargin=2.5 cm,tmargin=3.5cm,bmargin=2.5cm,
paper=a4paper]{geometry}
\newcommand{\Bk}{\color{black}}

\newcommand{\Ab}{\mathbf A}

\newcommand{\Fb}{\mathbf F}

\newcommand{\R}{\mathbb R}
\newcommand{\C}{\mathbb C}
\newcommand{\gse}{\mathrm E_{\rm gs}(\kappa,H,\varepsilon)}

\DeclareMathOperator{\curl}{curl}

\newtheorem{thm}{Theorem}[section]
\newtheorem{prop}[thm]{Proposition}
\newtheorem{lem}[thm]{Lemma}

\theoremstyle{remark}

\newtheorem{rem}[thm]{Remark}

\newcommand{\eq}{\begin{equation}}
\newcommand{\eeq}{\end{equation}}

\def\curl{\text{\rm curl\,}}

\def\gse{\mathrm E}

\def\C{\bold C}

\def\R{\mathbb  R}

\def\Z{\mathbb Z}
\def\0{\bold 0}

\def\mA{\mathcal A}

\def\C{\mathbb C}

\numberwithin{equation}{section}

\title[Magnetic Laplacian]
{Averaging of magnetic fields and applications}

\author{Ayman Kachmar}
\author{Mohammad Wehbe}
\address{Department of Mathematics, Lebanese University, Nabatieh, Lebanon.}
\email{akachmar@ul.edu.lb}
\email{wehbewehbe@gmail.com}

\date{\today}

\begin{document}
\begin{abstract}
We estimate the magnetic Laplacian energy norm in appropriate planar domains under a weak regularity hypothesis on the magnetic field. Our main contribution is an averaging estimate, valid in small cells, allowing  us to pass from  non-uniform to uniform magnetic fields. As a matter of application,  we derive new upper and lower bounds  of the lowest eigenvalue  of the Dirichlet Laplacian
which match in the regime of large  magnetic field intensity. Furthermore, our averaging technique allows us to estimate the non-linear Ginzburg-Landau energy, and as a byproduct,  yields a non-Gaussian trial state for the Dirichlet magnetic Laplacian. 
\end{abstract}
\maketitle

\tableofcontents

\section{Introduction}
The spectral properties of magnetic Schr\"odinger operators with minimal regularity assumptions on the magnetic field, magnetic potential and electrical potential, have been central since decades \cite{LS}. Averaging of magnetic fields was also a valuable tool to study such spectral properties, notably the question of existence of a compact resolvent \cite{I}. 

In this paper, we study the averaging of magnetic fields in the context of spectral asymptotics (large field/semi-classical asymptotics). Our estimates will  allow us to capture the leading order term in the large field asymptotics for the ground state energy of the magnetic Laplacian with a Dirichlet condition, via the essential infimum of the scalar magnetic field, under a weak regularity hypothesis.

\subsection{The magnetic field}
Consider a real-valued function 
\begin{equation}\label{eq:B}
B\in H^1(\R^2)\,.
\end{equation}
The function $B$   stands for a  magnetic field (more precisely this is the vertical magnetic field with non-uniform intensity $B$, i.e. $B\vec z$).  
We introduce the corresponding magnetic potential $\Ab$ as follows
\begin{equation}\label{eq:mag-pot}
\Ab(x)=(A_1(x),A_2(x)):=2\int_0^1 B(sx)\Ab_0(sx)ds\qquad (x\in\R^2)\,,
\end{equation}
where $\Ab_0$ is the \emph{canonical} magnetic potential, satisfying $\curl\Ab_0=1$ and  defined as follows
\begin{equation}\label{eq:A0}
\Ab_0(x)=\frac12(-x_2,x_1) \quad (x=(x_1,x_2)\in\R^2)\,.
\end{equation}
Clearly, $\Ab\in H^2(\R^2)$ and 
\begin{equation}\label{eq:curl-A}
\curl\Ab:=\partial_{x_1}A_2-\partial_{x_2}A_1=B\quad{\rm in ~}\R^2\,.
\end{equation}
There are many other reasonable choices  for the magnetic potential generating the magnetic field $B$, e.g. $\Ab+\nabla\chi$ for any smooth function $\chi$. 

The aim of this paper is to estimate quantities of the form
\begin{equation}\label{eq:qf-g}
\int_U |(\nabla-i\sigma\Ab)u|^2\,dx
\end{equation}
where $\sigma\in\R$, $U$ is an appropriate convex subset of $\R^2$, typically a square or a disc of small diameter compared to the parameter $\sigma$, and $u\in H^1(U)$. Such questions naturally occur in many problems of mathematical physics, such as superconductivity \cite{FH-b}, liquid crystals \cite{FKP} and the theory of Schr\"odinger operators \cite{R}. The case of a smooth $\Ab$ is well developed in the literature, so our aim here is to address this question for the less regular case where   $B\in H^1(\R^2)$ (i.e. $\Ab\in H^2(\R^2)$). This is related to \cite[Sec.~16.6.1, Open Problem~9]{FH-b} and \cite[Problem~2.2.9]{Pa8}. 

Our approach to approximate the quantity in \eqref{eq:qf-g} is through an averaging technique which will allow us to pass from $\Ab$ 
generating the non-smooth field $B$, to $\Ab_{\rm av}$ generating a constant field $B_{\rm av}$. The approximation will be valid in the regime of large field intensity, $\sigma\to+\infty$, and small domain $U$, ${\rm diam}(U)\to0$ (${\rm  diam}(U)$ stands for the \emph{diameter} of $U$). The precise statement will be given in Theorem~\ref{cor-avg} and Proposition~\ref{prop:app-qf} below.

\subsection{The averaging estimate}

Assume that
\begin{equation}\label{eq:U}
x_0\in U\subset\R^2~\text{and}~U\text{ is  open,  convex and  bounded}\,.
\end{equation}
We denote by $|U|$ the area of $U$, and by ${\rm diam}(U)$, the diameter of $U$. 
We introduce the new magnetic potential
\begin{equation}\label{eq:A-new}
\Ab_{\rm new}^U(x)=2\int_0^1 B(s(x-x_0)+x_0)\Ab_0\big(s(x-x_0)\big)\,ds\,,
\end{equation}
where $\Ab_0$ is the canonical magnetic potential introduced in \eqref{eq:A0}.

Note that, on $U$, $\curl\Ab_{\rm new}^U=B=\curl\Ab$, where $\Ab$ is the magnetic potential in \eqref{eq:mag-pot}.  So there exists a function $\varphi^U\in H^1(U)$ such that
\begin{equation}\label{eq:A-new=A}
\Ab=\Ab_{\rm new}^U-\nabla\varphi^U\quad{\rm on~}U\,.
\end{equation}
We introduce the average of the magnetic field $B$ in $U$ as follows
\begin{equation}\label{eq:B-av}
B_{\rm av}^U=\frac1{|U|}\int_U B(x)\,dx\,.
\end{equation}
It is then natural to introduce  the \emph{average} magnetic potential 
\begin{equation}\label{eq:A-av}
\Ab_{\rm av}^U(x)=B_{\rm av}^U\Ab_0(x-x_0)=2B_{\rm av}^U\int_0^1 \Ab_0\big(s(x-x_0)\big)\,ds\,,
\end{equation}
which generates the constant averaged magnetic field,  $\curl\Ab_{\rm av}^U=B_{\rm av}^U$. Theorem~\ref{cor-avg} below establishes that  the magnetic potential $\Ab_{\rm av}^U$ is a good approximation of $\Ab_{\rm new}$ in the convex domain $U$.

\begin{thm}\label{cor-avg}
Let $B\in H^1(\R^2)$. For  every  domain $U\subset\R^2$ satisfying \eqref{eq:U}, the following inequality holds,
$$ \int_U|\Ab_{\rm new}^U(x)-\Ab_{\rm av}^U(x)|^2\,dx\leq 8\delta^4\|\nabla B\|_{L^2(U)}^2\,,$$ 
where $\delta={\rm diam}(U)$,  $\Ab_{\rm new}^U$ and $\Ab_{\rm av}^U$ are introduced in \eqref{eq:A-new} and \eqref{eq:A-av} respectively.
\end{thm}

\subsection{The Dirichlet magnetic Laplacian}

As a consequence of Theorem~\ref{cor-avg}, we can estimate the lowest eigenvalue,  $\lambda(\sigma,\Ab;\Omega)$, of the Dirichlet magnetic Laplacian
$-(\nabla-i\sigma\Ab)^2$ in $L^2(\Omega)$, for a   domain $\Omega$  with a  smooth $C^1$ boundary. Studying the strong field asymptotics,  the essential infimum of  the function $B$ in $\Omega$ shows up; this is the quantity introduced as follows
\begin{equation}\label{eq:ess-inf}
m_0(B;\Omega):={\rm ess}\inf_{\hskip-0.5cm x\in\Omega} B(x)=\sup\{c\in\R~:~B(x)\geq c~{\rm a.e.~on~}\Omega\}\,.
\end{equation}
 The variational min-max principle allows us to express the eigenvalue as follows\,\footnote{The definition of the eigenvalue $\lambda(\sigma,\Ab;\Omega)$ requires a vector field $\Ab$ (and consequently a magnetic field $B$) defined on $\Omega$, not the whole space $\R^2$. Our assumption on the domain $\Omega$ allows us to extend  functions in the Sobolev space $H^1(\Omega)$ to functions in the space $H^1(\R^2)$, so that starting with $B\in H^1(\R^2)$ is not really a restriction. Our proofs require to deal with the value of  the magnetic field  outside the set $\Omega$.}  (when $\Omega$ is bounded)
\begin{equation}\label{eq:evD}
\lambda(\sigma,\Ab;\Omega)=\inf_{u\in  H^1_0(\Omega)\setminus\{0\}}
\frac{\|(\nabla-i\sigma\Ab)u\|^2_{L^2(\Omega)}}{\|u\|^2_{L^2(\Omega)}}\,.
\end{equation} 
Now we state our \emph{new} estimates on the eigenvalue $\lambda(\sigma,\Ab;\Omega)$.
\begin{thm}\label{thm:evD}
Assume that $\Omega=\bigcup\limits_{i=1}^N \Omega_i$ where $N\geq 1$ is a positive integer, the sets $\overline{\Omega_i}$ are pairwise disjoint, and each $\Omega_i$ is a bounded  connected domain of $\R^2$ such that $\partial\Omega_i$   consists  of a finite number of smooth $C^1$ closed curves.  

If $B\in H^1(\Omega)$ and the essential infimum in \eqref{eq:ess-inf} is {\bf positive}, then the lowest eigenvalue in \eqref{eq:evD} satisfies
\[m_0(B;\Omega)\sigma\leq   \lambda(\sigma,\Ab;\Omega)\leq m_0(B;\Omega)\sigma+o(\sigma)\quad\big(\sigma\to+\infty\big)\,.\]
\end{thm}
The content of Theorem~\ref{thm:evD} is consistent with the known estimates for a smooth magnetic field  (see \cite{HM}), in which case the essential infimum becomes
\[m_0(B;\Omega)=\min_{x\in\overline{\Omega}}B(x)\,,\]
and the remainder term $o(\sigma)$ can be explicitly controlled.

The  non-asymptotic lower bound,  $\lambda(\sigma,\Ab;\Omega)\geq  \sigma\, m_0(B;\Omega)$, follows by a standard argument. The matching upper bound,
$\lambda(\sigma,\Ab;\Omega)\leq m_0(B;\Omega)\sigma+o(\sigma)$, follows by constructing a trial state; the produced errors are controlled by the averaging estimate of Theorem~\ref{cor-avg}.

The novelty in Theorem~\ref{thm:evD} is establishing its validity in the weakly regular situation when \eqref{eq:B} holds. This prevents us of deducing it from other works treating non-uniform magnetic fields, like smooth magnetic fields \cite{Att, HM, M,  R}, $B\in C^{0,\alpha}(\R^2)$, or step magnetic fields \cite{A, HPRS}. 

It would be desirable to establish Theorem~\ref{thm:evD} under the much weaker hypothesis, $B\in L^2(\Omega)$. This is motivated by the current Theorem~\ref{thm:evD} and the existing results when $B$ is a step function \cite{A, HPRS}. However, knowing  $B\in L^2(\Omega)$ without further regularity, our averaging estimate in Theorem~\ref{cor-avg} will be out of reach, thereby preventing us from proving the upper bound in Theorem~\ref{thm:evD} without the additional property $\nabla B\in L^2(\Omega;\R^2)$.

\subsection{The Ginzburg-Landau functional}

Our averaging mechanism is robust in the study of the non-linear Ginzburg-Landau functional (see Theorem~\ref{thm:GL} below), which also contributes to the proof of Theorem~\ref{thm:evD} (by providing us with a useful trial state). Under the regularity assumption \eqref{eq:B} on $B$, our contribution adds to the mainstream of understanding the role of non-uniform  magnetic fields in the Ginzburg-Landau model \cite{Att, AKP, HK-arma, HK-cvpde, PK}. Handling the particularities of our regularity hypothesis in \eqref{eq:B} would not be possible without the averaging estimate of Theorem~\ref{cor-avg}.

We restrict our study to a bounded domain $\Omega\subset\R^2$ which we assume  connected and with a smooth boundary consisting of a finite number of smooth curves of class $C^1$. More precisely, we assume that $\Omega=\tilde\Omega\setminus \bigcup\limits_{k=1}^n \omega_k$, where $\omega_1,\cdots,\omega_n,\tilde\Omega$ are simply connected domains with smooth $C^1$ boundaries, each $\omega_k\subset\tilde\Omega$, and the sets $\overline{\omega_k}$ are pair-wise disjoint.

A central role will be played by    the  magnetic potential $\Fb\in H^2(\tilde\Omega)$ satisfying
\begin{equation}\label{eq:F}
\curl\Fb=B\,,\quad {\rm div}\Fb=0\quad {\rm in}~\tilde\Omega\,,\quad 
\nu\cdot\Fb=0\quad{\rm on~}\partial\tilde\Omega\,,  
\end{equation}
where $\nu$ is the unit interior normal vector of $\partial\tilde\Omega$. Since the domain $\tilde\Omega$ is simply connected,  $\Ab-\Fb$ is a gradient field on $\tilde\Omega$ and we can find a  function $\vartheta\in H^3(\tilde\Omega)$ such that  (see \cite[Prop.~D.1.1]{FH-b})
\begin{equation}\label{eq:A=F}
\Ab=\Fb+\nabla \vartheta\quad{\rm on~}\tilde\Omega\,.
\end{equation}

\subsubsection*{The functional \& critical configurations}\

The GL  functional is defined for configurations $(\psi,\mA)$ in the space $H^1(\Omega;\C)\times H^1(\tilde\Omega;\R^2)$ as follows
\begin{equation}\label{eq:GL}
\mathcal G(\psi,\mA)=\int_\Omega\left(|(\nabla-i\kappa H\mA)\psi|^2-\kappa^2|\psi|^2+\frac{\kappa^2}{2}|\psi|^4\right)\,dx+(\kappa H)^2\int_{\tilde\Omega}|\curl(\mA-\Fb)|^2dx\,,
\end{equation}
where $\Fb$ is the magnetic potential introduced in \eqref{eq:F}.
We introduce the ground state energy
\begin{equation}\label{eq:gseGL}
\gse (\kappa,H)=
\inf\{\mathcal G(\psi,\mA)~:~(\psi,\mA)\in H^1(\Omega;\C)\times H^1_{\rm div}(\tilde\Omega;\R^2)\}\,,
\end{equation}
where $\mathcal A\in H^1_{\rm div}(\tilde\Omega;\R^2)$ means
\begin{equation}\label{eq:mA}
\mA\in H^1(\tilde\Omega;\R^2)\,,\quad  {\rm div}\mA=0\quad {\rm in}~\tilde\Omega\,,\quad 
\nu\cdot\mA=0\quad{\rm on~}\partial\tilde\Omega\,,
\end{equation}
and $\nu$ is the inward normal vector of $\partial\tilde\Omega$. The property of gauge invariance yields \cite[Sec.~10.1.2]{FH-b}
\[\gse (\kappa,H)=
\inf\{\mathcal G(\psi,\mA)~:~(\psi,\mA)\in H^1(\Omega;\C)\times H^1(\tilde\Omega;\R^2)\}\,.\]

Every minimizing configuration $(\psi,\mathcal A)_{\kappa,H}$ is a critical point of the GL functional, that is it satisfies the following equations:
\begin{equation}\label{eq:Euler-Lag}
\left\{
\begin{array}{lll}
-(\nabla-i\kappa H\mathcal A)^2\psi=\kappa^2(1-|\psi|^2)\psi &\text{in}\:\:\Omega,\\
-\nabla^\bot\Big(\text{curl}(\mathcal A-\Fb)\Big)=\displaystyle\frac{1}{\kappa H}\mathbf 1_{\Omega}\,\text{Im}\:\Big(\overline{\psi}(\nabla -i\kappa H\mathcal A)\psi\Big)&\text{in}\:\: \tilde\Omega\,,\\
\nu\cdot (\nabla-i\kappa H\mathcal A)\psi=0&\text{on}\:\:\partial\Omega,\\
\text{curl}(\mathcal A-\Fb)=0 &\text{on}\:\:\partial\tilde \Omega,
\end{array}\right.
\end{equation}
where $\nabla^\bot=(\partial_{x_2},-\partial_{x_1})$ is the Hodge gradient.\Bk

\subsubsection*{Bulk energy function}\

The GL ground state energy $\gse(\kappa,H)$ in \eqref{eq:gseGL} is closely related to a simplified effective energy, which we will call the bulk energy function. This is the \emph{ concave} function $g:[0,+\infty)\to[-\frac12,0]$ that we will introduce below. First, we set $g(0)=-\frac12$ and $g(b)=0$ for all $b\geq 1$; the definition of $g(b)$ when $b\in(0,1)$ is implicit through the large area limit of a certain non-linear energy \cite{AS, FK-cpde, SS02}.

Let $R>0$ and $Q_R=(-R/2,R/2)\times(-R/2,R/2)$. We define the following Ginzburg-Lamdau energy with the constant magnetic field on $H^1(Q_R)$ by
$$G_{b,Q_R}(u)=\int_{Q_R}\left(b|(\nabla-i {\bf A_0})u|^2-|u|^2+\frac{1}{2}|u|^4\right)dx.$$
Here  ${\bf A_0}$ is the vector field introduced in \eqref{eq:A0}.
We introduce  the two ground state energies
\[
m_0(b,R)=\inf_{u\in H^1_0(Q_R)}G_{b,Q_R}(u),\quad{\rm and}\quad m(b,R)=\inf_{u\in H^1(Q_R)}G_{b,Q_R}(u).
\]
We gather the following remarkable properties (see \cite[Thm.~2.1]{FK-cpde}):
\begin{itemize}
\item If $b\geq 1$ and $R>0$, then $m_0(b,R)=0$\,.
\item $m_0(0,R)=-\frac{R^2}2$\,.
\item Every minimizer $u_{b,R}$ of $m_0(b,R)$ or $m(b,R)$ satisfies the uniform bound $|u_{b,R}|\leq 1$.
\item  For all $b\in [0,\infty)$, the following limits exist 
\[g(b)=\lim_{R\rightarrow\infty}\frac{m_0(b,R)}{R^2}=\lim_{R\rightarrow\infty}\frac{m(b,R)}{R^2}.\]
\item There exist positive constants $C$ and $R_0$, such that, for all $R\geq R_0$ and $b\in [0,1)$,
\begin{equation}\label{eq:g(b)}
g(b)\leq \frac{m_0(b,R)}{R^2}\leq g(b)+\frac{C}{R}\quad\text{and}\quad g(b)-\frac{C}{R}\leq\frac{m(b,R)}{R^2}\leq g(b)+\frac{C}{R}.\end{equation}
\end{itemize}

\subsubsection*{The leading order energy}\

The approximation of the energy $\gse(\kappa,H)$ will require the decomposition of the domain $\Omega$ into small cells, which we describe below and eventually define the leading order energy in \eqref{eq:E-asy}. 

We fix two positive constants $c_1$ and $c_2$ such that $0<c_1<c_2$, and we let $\ell$ be a parameter that varies in the following manner
\begin{equation}\label{eq:ell}
c_1\kappa^{-3/4}\leq \ell\leq c_2\kappa^{-3/4}\,,
\end{equation}
so that $\ell$ approaches $0$ in the regime of large GL parameter $\kappa$. 

Now we set
\begin{equation}\label{eq:x-m,n}
x_{m,n}^\ell :=(\ell m,\ell n)\qquad \big((m,n)\in\Z^2\big)\,,
\end{equation}
\begin{equation}\label{eq:J} 
\mathcal J_\ell=\{x_{m,n}^\ell~:~(m,n)\in\Z^2~\&~ Q_\ell(x_{m,n}^\ell)\subset\Omega\}\,,
\end{equation}
and 
\begin{equation}\label{eq:Om-ell}
\Omega_\ell:=\bigcup_{x\in\mathcal J_\ell}Q_\ell(x)\,,
\end{equation} 
where $Q_\ell(\cdot)$ is the open square introduced in \eqref{eq:Q-ell}. 
The definition of the set $\mathcal J_\ell$ yields that the squares  $\big(Q_\ell(x)\big)_{x\in\mathcal J_\ell}$ are pairwise disjoint, and    $\Omega_\ell\subset\Omega$.  Consequently  the set $\mathcal J_\ell$ is finite, since the domain $\Omega$ is bounded, and its cardinal
\begin{equation}\label{eq:N}
N(\ell):={\rm Card}(\mathcal J_\ell)
\end{equation}
satisfies the obvious upper bound
\begin{equation}\label{eq:N-ub}
N(\ell)\leq |\Omega|\ell^{-2}\,.
\end{equation}
Furthermore, by smoothness and boundedness of the boundary $\partial\Omega$, we can write the following lower bound on the number $N(\ell)$, 
\begin{equation}\label{eq:N-ell=0}
N(\ell)\geq |\Omega|\ell^{-2}-\mathcal O(\ell^{-1})\qquad(\ell\to0_+)\,.
\end{equation}

We demonstrate in Theorem~\ref{thm:GL} below  that the  GL ground state energy, $\gse(\kappa,H)$, introduced in \eqref{eq:gseGL}, is to leading order given by the following energy
\begin{equation}\label{eq:E-asy}
\mathrm E^{\rm asy}(b,\ell)=\ell^2\sum_{x\in\mathcal J_\ell}g\big(bB_{\rm av}^\ell(x)\big)\,,
\end{equation}
where
 $\ell$ and $\mathcal J_\ell$ are introduced in \eqref{eq:ell} and \eqref{eq:J} respectively, and $g(\cdot)$ is the bulk energy function introduced in \eqref{eq:g(b)}.

\begin{thm}\label{thm:GL}
Assume that there exists a positive real number $c$ such that $B\geq c>0$ a.e. in $\Omega$. Given $\epsilon\in(0,1)$ and $c_2>c_1>0$, there exist constants $C,\kappa_0$ such that, for all $\kappa\geq \kappa_0$, $H=b\kappa$, $\ell$ satisfying \eqref{eq:ell},  and $b\in(\epsilon,\epsilon^{-1})$, the following holds
\[
\left|\gse(\kappa,H) -\kappa^2\mathrm E^{\rm asy}(b,\ell) \right|\leq C\kappa^{15/8}\,.
\]
\end{thm}

\begin{rem}\label{rem:GL}
Since $g(\cdot)\geq -\frac12$, we get by \eqref{eq:N-ub},
$$-\frac12|\Omega|\leq \ell^2\sum_{x\in\mathcal J_\ell}g\big(bB_{\rm av}^\ell(x)\big)\leq 0\,.$$
Furthermore, since $g(\cdot)$ is { concave}, { $-g(\cdot)$ is convex} and Jensen's inequality yields
$$\ell^2\sum_{x\in\mathcal J_\ell}g\big(bB_{\rm av}^\ell(x)\big)\geq \sum_{x\in\mathcal J_\ell}\int_{Q_\ell(x)} g\big( bB(y)\big)\,dy=\int_{\Omega}g\big(bB(y)\big)\,dy+\mathcal O(\ell)\,.$$
Consequently, we see that, if $b>0$ is a fixed constant (independent from the parameters $\kappa,H,\ell$), the effective energy in \eqref{eq:E-asy}, satisfies {(see Remark~\ref{rem:bB>1a.e.} for additional details)}
\begin{equation}\label{eq:l-o-t=0}
\mathrm E^{\rm asy}(b,\ell)\underset{\ell\to0_+}{=}o(1)\Longleftrightarrow \big|\{y\in\Omega~:~bB(y)<1\}\big| =0\,.
\end{equation}
\end{rem}

\begin{rem}\label{rem:GL*}
We can deduce the eigenvalue upper bound mentioned in Theorem~\ref{thm:evD} from Theorem~\ref{thm:GL}, by using the GL order parameter as a trial state for the Dirichlet eigenvalue.  We present this construction in Sec.~\ref{sec:mag-lap}, which highlights the possibility of extracting spectral asymototics from the study of the  GL model, despite the many existing results that go in the opposite direction, namely studying the GL model starting from eigenvalue estimates of the magnetic Laplacian. 
\end{rem}

\subsection{Organization of the paper}

The paper is organized as follows. Section~\ref{sec:p} contains some standard material that we are going to use through the paper.   Section~\ref{sec:av} contains the proof of the averaging estimate, Theorem~\ref{cor-avg}. The   estimate of the energy in \eqref{eq:qf-g} occupies Section~\ref{sec:qf}. The proof of Theorem~\ref{thm:evD} is given in Section~\ref{sec:mag-lap*}. Section~\ref{sec:GL} is devoted to the study of the Ginzburg-Landau model and ends up by an alternative proof of the eigenvalue upper bound for the Dirichlet magnetic Laplacian 
(Sec.~\ref{sec:mag-lap}).

\section{Preliminaries and notation}\label{sec:p}

The purpose of this section is to introduce the necessary material for the statement of the main theorems  in the subsequent sections.

\subsection*{Asymptotic order}

We will use the standard Landau notation to denote bounded quantities, $\mathcal O(1)$, and vanishingly small quantities, $o(1)$, with respect to a parameter $\sigma$ living in a neighborhood of $+\infty$. Additionally, we use the notation $\approx $ in the following context; given two functions $a(\sigma)$ and $b(\sigma)$, writing  $a\approx b$ means that there exist positive constants $\sigma_0,c_1,c_2$ such that $c_1 b(\sigma)\leq a(\sigma)\leq c_2b(\sigma)$. We use the letter $C$ to denote constants. The value of $C$ might change from one inequality to another without mentioning this explicitly.

\subsection*{The averaged magnetic field}

For all $x\in\R^2$ and $\ell>0$, we introduce the \emph{open} square of center $x$ and side-length $\ell$ as follows
\begin{equation}\label{eq:Q-ell}
Q_\ell(x)=(x-\ell/2,x+\ell/2)\times (x-\ell/2,x+\ell/2)\,.
\end{equation}
We introduce the averaged magnetic field in the square $Q_\ell(x)$,
\begin{equation}\label{eq:mf-av-l}
B_{\rm av}^\ell(x)=\frac1{\ell^2}\int_{Q_\ell(x)}B(y)\,dy\,.
\end{equation}
Note that, if $B$ satisfies the following condition in some open set $\Omega\subset\R^2$,
\begin{equation}\label{eq:B>c}
\exists\,c\in\R\,,\quad B\geq c\quad {\rm a.e.}
\end{equation}
then the averaged magnetic field satisfies
\begin{equation}\label{eq:B-av>c}
B_{\rm av}^\ell(x)\geq c\quad  {\rm whenever }~Q_\ell(x)\subset\Omega\,.
\end{equation}
Assuming \eqref{eq:B}, we will prove that $B_{\rm av}^\ell(x)$ can have only \emph{slow} growth in the small length limit.

 \begin{lem}\label{lem1}
For all  $\zeta\in(0,\frac12]$, there exist $C,\ell_0>0$ such that,   
for all $\ell\in (0,\ell_0)$, $B\in H^1(\R^2)$ and $x\in\R^2$, the following holds,
\[
|B_{\rm av}^\ell(x)|\leq C\ell^{-2\zeta}\|B\|_{H^1(\R^2)}\,.
\]
\end{lem}
 \begin{proof}
 Notice that,
 \[
 \big| B_{\rm av}^\ell(x)\big|\leq
  \frac{1}{\ell^2}\int_{Q_\ell(x)}|B(y)|dy\,.
 \]
 Let $p=\frac{1}{\zeta}$ and $q=\frac{p}{p-1}$ the H\"older conjugate of $p$. By H\"older's inequality
 $$\int_{Q_\ell(x)}|B(y)|dy\leq 
 |{Q_\ell(x)}|^{1/q} \|B\|_{L^p(\R^2)}\,.
 $$
 Consequently, 
$$\big| B_{\rm av}^\ell(x)\big|\leq  
  \ell^{\frac2q-2} \|B\|_{L^p(\R^2)}
  =\ell ^{-2\zeta} \|B\|_{L^p(\R^2)}\,.
  $$
To finish the proof, we { note that $p\geq 2$} and use the Sobolev embedding of $H^1(\R^2)$ in $L^p(\R^2)$. 
 \end{proof}

\section{Averaging of the magnetic field}\label{sec:av}

The proof of Theorem~\ref{cor-avg} relies on the following proposition.

\begin{prop}\label{prop:B(y)=B(x)}
For every $s\in(0,1)$,  and every  domain $U\subset\R^2$ satisfying \eqref{eq:U}, the following inequality holds,
$$s^2\int_U |B_{s,x_0}(x)-B_{\rm av}^U|^2\,dx\leq 8\delta^2\|\nabla B\|_{L^2(U)}^2\,,$$
where $\delta={\rm diam}(U)$, $B\in H^1(\R^2)$,   $B_{\rm av}^U$ is  introduced in \eqref{eq:B-av} and, for any $x\in U$, 
\[ B_{s,x_0}(x):=B\big(s(x-x_0)+x_0\big)\,.\]
\end{prop}
\begin{proof}
We will prove Proposition~\ref{prop:B(y)=B(x)} in the special case where $B\in C^1(\R^2)$. The general case follows then by a density argument, using the density of $C^\infty(\R^2)$ in $H^1(\R^2)$ and the Sobolev embedding of $H^1(\R^2)$ in $L^4(\R^2)$.

We start by  noticing that
$$B_{s,x_0}(x)-B_{\rm av}^U=\frac1{|U|}\int_U \big(B_{s,x_0}(x)-B(y)\big)\,dy\,.$$
By Jensen's inequality,
$$\int_U |B_{s,x_0}(x)-B_{\rm av}^U|^2\,dx\leq \frac1{|U|}\int_U \left(\int_U |B_{s,x_0}(x)-B(y)|^2 \,dy\right)dx\,.$$
Now, it is enough to prove the following inequality, 
\begin{equation}\label{eq:R-B-Bav}
\int_U\int_{U}|B_{s,x_0}(x)-B(y)|^2dxdy \leq 8\delta^2|U|\|\nabla B\|_{L^2(U)}^2\,.
\end{equation}
Indeed,  for $y,z\in U$ with $y\not=z$, the  convexity of $U$  ensures that $z+t\frac{y-z}{|y-z|}\in U$ for  $t\in[0,|y-z|]$, hence,
\begin{align*}
    B(z)-B(y)&=-\int_0^{|y-z|}\dfrac{d}{dt}B\left(z+t\frac{y-z}{|y-z|}\right)dt\\
    &=-\int_0^{|y-z|}\nabla B\left(z+t\frac{y-z}{|y-z|}\right)\cdot\dfrac{y-z}{|y-z|}\,dt\,.
    \end{align*}
    Consequently,
 \begin{align*}
   | B(z)-B(y)|&\leq \int_0^{|y-z|}\left|\nabla B\left(z+t\frac{y-z}{|y-z|}\right)\right|\,dt\\
   &=|y-z|\int_0^1 \left|\nabla B\left(z+\tau (y-z)\right)\right|\,d\tau
   \end{align*}
   after performing the change of variable $\tau=t/|y-z|$. By Jensen's inequality, we get further
 \begin{align*}
   | B(z)-B(y)|^2&\leq |y-z|^2\int_0^1\left|\nabla B\left(z+\tau (y-z)\right)\right|^2\,d\tau\\
   &\leq \delta^2  \int_0^1\left|\nabla B\left(z+\tau (y-z)\right)\right|^2\,d\tau\,.
   \end{align*}
   We use the foregoing inequality for $z=m_{s,x_0}(x):=s(x-x_0)+x_0$,  $x\in U$, from which we get 
   \begin{equation}\label{eq:R-Bs-B}
   |B_{s,x_0}(x)-B(y)|^2\leq \delta^2\Big( I_1(x,y;s)+I_2(x,y;s)\Big)\,,
   \end{equation}
   where
   \[ I_1(x,y;s):=\int_0^{1/2}\left|\nabla B\left(m_{s,x_0}(x)+\tau (y-m_{s,x_0}(x))\right)\right|^2\,d\tau \]
   and
   \[ I_2(x,y;s):= \int_{1/2}^{1}\left|\nabla B\left(m_{s,x_0}(x)+\tau (y-m_{s,x_0}(x))\right)\right|^2\,d\tau\,.\]
   For $\tau\in[0,1]$ and $y\in U$, consider the set  $U_{y,\tau}=\{m_{s,x_0}(x)+\tau\big(y-m_{s,x_0}(x)\big)~:~{ x\in U}\}$; since $U$ is  convex, we observe that $U_{y,\tau}\subset U$. Now, performing the change of variable 
   \[x\mapsto a:=m_{s,x_0}(x)+\tau (y-m_{s,x_0}(x))\,,\] we get for all $s\in(0,1)$ and $\tau\in[0,\frac12]$,
   \[ 
    \int_U \left|\nabla B\left(m_{s,x_0}(x)+\tau (y-m_{s,x_0}(x))\right)\right|^2dx
    =\frac1{s^2(1-\tau^2)}\int_{U_{y,\tau}}|\nabla B(a)|^2\,da\\
    \leq \frac4{s^2}\|\nabla B\|_{L^2(U)}^2\,.
    \]
  Integrating again with respect to $y\in U$, we get
  \begin{equation}\label{eq:R-est-I1}
   \int_U \left(\int_U I_1(x,y;s)dx\right)dy\leq \frac4{s^2} |U|\|\nabla B\|_{L^2(U)}^2\,.
  \end{equation} 
  We estimate the integral of $I_2$ in a similar fashion. Doing the change of variable $y\mapsto \tilde a:=m_{s,x_0}(x)+\tau (y-m_{s,x_0}(x)) $, we observe for all $\tau\in[\frac12,1]$, 
\[   \int_U \left|\nabla B\left(m_{s,x_0}(x)+\tau (y-m_{s,x_0}(x))\right)\right|^2dy
    =\frac1{\tau^2}\int_{V_{x,\tau}}|\nabla B(\tilde a)|^2\,d\tilde a\leq 4\|\nabla B\|_{L^2(U)}^2\,,\]  
    where $V_{x,\tau}:=\{m_{s,x_0}(x)+\tau\big(y-m_{s,x_0}(x)\big)~:~y\in U\}\subset U$, since $U$ is convex. After integrating with respect to $x\in U$, we get 
     \begin{equation}\label{eq:R-est-I2}
  \int_U \left(\int_U I_2(x,y;s)dy\right)dx\leq 4|U|\|\nabla B\|_{L^2(U)}^2\,.
  \end{equation} 
  Inserting \eqref{eq:R-est-I1} and \eqref{eq:R-est-I2} into \eqref{eq:R-Bs-B}, we get eventually \eqref{eq:R-B-Bav}, which finishes the proof of Proposition~\ref{prop:B(y)=B(x)}.
 \end{proof}

\begin{proof}[Proof of Theorem~\ref{cor-avg}]
Collecting \eqref{eq:A-new} and \eqref{eq:A-av}, we write, for all $x\in U$,
\[\Ab_{\rm new}^U(x)-\Ab_{\rm av}^U(x)=2\int_0^1\Big( B(s(x-x_0)+x_0) - B_{\rm av}^U\Big)\Ab_0\big(s(x-x_0)\big)\,ds\,.\]
Since $ \big|\Ab_0\big(s(x-x_0)\big)\big|\leq \frac12s|x-x_0|\leq \frac12s\,{\rm diam}(U)$ on $U$, we get by using Jensen's inequality,
\begin{equation}\label{eq:A-Aav}
\forall\,x\in U\,,\quad |\Ab_{\rm new}^U(x)-\Ab_{\rm av}^U(x)|^2\leq \delta^2\int_0^1 |B(s(x-x_0)+x_0)-B_{\rm av}^U|^2 s^2\,ds\,.
\end{equation} 
We apply Proposition~\ref{prop:B(y)=B(x)} to estimate the term in the r.h.s. in \eqref{eq:A-Aav}. This finishes the proof of Theorem~\ref{cor-avg}.
\end{proof}

\begin{rem}\label{rem:cor-J}
If we 
perform the change of variable $y=s(x-x_0)+x_0$ and note that $U$ is convex (which guarantees that $y\in U$, for all $x\in U$), we deduce from \eqref{eq:A-Aav},
\begin{equation}\label{eq:A-Aav*}
\begin{aligned}
\int_U|\Ab_{\rm new}^U(x)-\Ab_{\rm av}^U(x)|^2\,dx&\leq \delta^2\int_0^1\int_U |B(s(x-x_0)+x_0)-B_{\rm av}^U|^2 s^2dx ds\\
&\leq \delta^2\int_U |B(y)-B_{\rm av}^U|^2 dy\,. 
\end{aligned}
\end{equation}
\end{rem}

\section{Approximation of the quadratic form}\label{sec:qf}

Given a bounded open set $U\subset\R^2$, a function $u\in H^1(U)$,  a vector field $\mathbf a\in H^1(U;\R^2)$ and a real number $\sigma$, we introduce  
\begin{equation}\label{eq:qf-loc}
q_\sigma\big(u,\mathbf a;U\big)=\int_{U} |(\nabla -i\sigma\mathbf a)u|^2\,dx\,.
\end{equation}

\begin{prop}\label{prop:app-qf}
Given $\eta,\rho\in(0,\frac12)$ and $0<c_1<c_2$, there exist  constants $C',\sigma_0>0$ such that the following is true. If
\begin{itemize}
\item  $\sigma\geq \sigma_0$\,;
\item $U\subset\R^2$ is open and convex\,;
\item $c_1\sigma^{-\rho}\leq {\rm diam}(U),|U|^{1/2}\leq c_2\sigma^{-\rho}$
\item $u\in H^1(U) \cap L^\infty(U)$,  $B\in H^1(\R^2)$ \& $\Ab$ defined by \eqref{eq:mag-pot}\,,
\end{itemize}
then there exists a function $\varphi:=\varphi^U\in H^1(U)$ such that 
\begin{multline*}
(1-\sigma^{-\eta})q_\sigma (v,\Ab_{\rm av}^U;U)-C'\sigma^{2-4\rho+\eta}\|\nabla B\|_{L^2(U)}^2{ \|u\|_{L^\infty(U)}^2}\\
\leq  q_\sigma(u,\Ab;U)\leq (1+\sigma^{-\eta})q_\sigma (v,\Ab_{\rm av}^U;U)+C'\sigma^{2-4\rho+\eta}\|\nabla B\|_{L^2(U)}^2{\|u\|_{L^\infty(U)}^2}
\end{multline*}
where $\Ab_{\rm av}^U$ is introduced in \eqref{eq:A-av} and $v=e^{i\sigma \varphi}u$.
\end{prop}
Later in the proof of Theorem~\ref{thm:evD}, we use the upper bound in Proposition~\ref{prop:app-qf} to compute the energy of a quasi-mode.

\begin{rem}\label{rem:scaling-lm}
The condition $\rho\in(0,\frac12)$ is a consequence of a scaling argument.  Since $x_0\in U$ and ${\rm diam}(U)\approx\sigma^{-\rho}$, we have $U\subset \{|x-x_0|\leq \mathcal O(\sigma^{-\rho})\}$. The change of variable, $y=\sigma^{1/2}(x-x_0)$ yields (see \eqref{eq:A-av})
\[q_\sigma (v,\Ab_{\rm av}^U;U)=\sigma\int_{\tilde U_\sigma}|(\nabla -B_{\rm av}^U\Ab_0)\tilde v|^2\,dy\,,\]
where $\tilde U_\sigma=\{y=\sigma^{1/2}(x-x_0),~x\in U\}\subset\{|y|\leq \mathcal O(\sigma^{\frac12-\rho})\}$ and $\tilde v(y)=v(x)$. To ensure that $\tilde U_\sigma$ approaches $\R^2$ (which is a fixed domain),  we impose the condition $\rho\in(0,\frac12)$.
\end{rem}

\begin{proof}[Proof of Proposition~\ref{prop:app-qf}]
Note that the following holds:
\begin{itemize}
\item[i.] (Gauge transformation) if  $v=e^{-i\sigma \phi}u$, then 
$q_{\sigma}(v,\Ab^U_{\rm av},U)=q_{\sigma}(v,\Ab^U_{\rm av}-\nabla\phi,U)$\,;
\item[ii.] (Cauchy's inequality) for every $a,b,\sigma>0$, 
$(a+b)^2\leq (1+\sigma^{-\eta})a^2+(1+\sigma^{\eta})b^2$\,;
\item[iii.] Theorem~\ref{cor-avg}\,;
\item[iv.] $\delta:={\rm diam}(U)$ satisfies $\delta^4\leq c_2^4\sigma^{-4\rho}$\,.
\end{itemize}
Now we write
  \begin{align*}
   q_{\sigma}(u,\Ab;U)&:=\int_U\left|(\nabla-i\sigma \Ab)u\right|^2dx\\
    &= \int_U\left|\left(\nabla-i\sigma (\Ab^U_{\rm new}-\Ab^U_{\rm av}+\Ab^U_{\rm av}-\nabla\phi)\right)u\right|^2dx\\
    &\overset{\rm i.}{=}\int_U\left|(\nabla-i\Ab^U_{\rm av})v-i\sigma(\Ab^U_{\rm new}-\Ab^U_{\rm av})u\right|^2dx\\
    &\overset{\rm ii.}{\leq} (1+\sigma^{-\eta})\int_U\left|(\nabla-i\sigma \Ab^U_{\rm av})v\right|^2dx+(1+\sigma^{\eta})\sigma^2\int_U\left|(\Ab^U_{\rm new}-\Ab^U_{\rm av})u\right|^2dx\\
    &\overset{\rm iii.}{\leq}  (1+\sigma^{-\eta})\int_U\left|(\nabla-i\sigma \Ab^U_{\rm av})v\right|^2dx+8\sigma^2(1+\sigma^{\eta})\delta^4\|\nabla B\|_{L^2(U)}^2\|u\|_{L^\infty(U)}^2\\
    &\overset{\rm iv.}{\leq}(1+\sigma^{-\eta})q_{\sigma}(v,\Ab^U_{\rm av};U)+8\sigma^2(1+\sigma^{\eta})c_2^4\sigma^{-4\rho}\|\nabla B\|_{L^2(U)}^2\|u\|_{L^\infty(U)}^2\\
    &\leq (1+\sigma^{-\eta})q_{\sigma}(v,\Ab^U_{\rm av};U)+C'\sigma^{2-4\rho+\eta}\|\nabla B\|_{L^2(U)}^2\|u\|_{L^\infty(U)}^2\,,
\end{align*}
with $C'=16c_2^4$. 
A similar argument yields 
$$(1-\sigma^{-\eta})q_\sigma (v,\Ab_{\rm av}^U;U)-C'\sigma^{2-4\rho+\eta}\|\nabla B\|_{L^2(U)}^2\|u\|_{L^\infty(U)}^2
\leq  q_\sigma(u,\Ab;U)\,.$$
\end{proof}

\section{Magnetic Laplacian}\label{sec:mag-lap*}

The aim of this section is to prove Theorem~\ref{thm:evD}, which is concerned with the principal eigenvalue of the magnetic Laplacian
\begin{equation}\label{eq:magLap}
\Delta_{\sigma \Ab}=-(\nabla-i\sigma\Ab)^2
\end{equation}
with domain (when $\Omega\subset\R^2$ is bounded and with a smooth $C^2$ boundary)
\begin{equation}\label{eq:dom}
\mathcal D=H^2(\Omega)\cap H^1_0(\Omega)\,.
\end{equation}
The operator $\Delta_{\sigma\Ab}$ is  self-adjoint  in the Hilbert space $L^2(\Omega)$ and its principal eigenvalue  is  introduced in \eqref{eq:evD}.

\subsection{Upper bound}

We will construct a trial state by means of a Gaussian function, but localized near a point $x_\varepsilon\in\Omega$ such that
the Lebesgue differentiation theorem holds for $B(x)$ and $|\nabla B(x)|$ at $x_\varepsilon$, and as $\varepsilon\to0_+$,
$B(x_\varepsilon)=m_0(B;\Omega)+\mathcal O(\varepsilon)$,
where $m_0(B;\Omega)$ is the essential infimum introduced in \eqref{eq:ess-inf}. 

 By the Lebesgue differentiation theorem, the two sets
\begin{align*}
&N=\{u\in\Omega,~\lim_{\ell\to0+}\frac1{|D(u,\ell)|} \int_{D(u,\ell)} |\nabla B(x)|^2dx\not= |\nabla B(u)|^2\}\\
&\tilde N= \{u\in\Omega,~\lim_{\ell\to0+}\frac1{|D(u,\ell)|} \int_{D(u,\ell)} B(x)dx\not= B(u)\}
\end{align*}
have zero Lebesgue measure, where  $D(u,\ell)$ denotes the open disk of center $u$ and radius $\ell$.

We assume that $m_0(B;\Omega)>0$. For all $\varepsilon\in(0,1]$, we introduce the set
\[M_\varepsilon=\{x\in\Omega,~ m_0(B;\Omega)\leq B(x)\leq m_0(B;\Omega)+\varepsilon\}\,.\]
Since the set $M_\varepsilon$ has a non-zero Lebesgue measure, $M_\varepsilon\not\subset N\cup\tilde N$, so we get by the Lebesgue differentiation theorem
\begin{multline}\label{eq:B-gB-av}
\exists\,x_\varepsilon\in M_\varepsilon\,,\quad \frac1{|D(x_\varepsilon,\ell)|} \int_{D(x_\varepsilon,\ell)} |\nabla B(x)|^2dx\underset{\ell\to0_+}{\longrightarrow} |\nabla B(x_\varepsilon)|^2<+\infty\\
{\rm and}\quad \quad \frac1{|D(x_\varepsilon,\ell)|} \int_{D(x_\varepsilon,\ell)} B(x) dx\underset{\ell\to0_+}{\longrightarrow} B(x_\varepsilon)<+\infty\,.
\end{multline}
In the sequel, $\rho\in(0,\frac12)$ and $U:=D(x_\varepsilon,\sigma^{-\rho})\subset\Omega$ for $\sigma$ sufficiently large. Let $\varphi:=\varphi^U$ be the gauge function in Proposition~\ref{prop:app-qf}. Consider the  trial state $u(x)=e^{-i\sigma \varphi}v(x)$, with
$v$ the following Gaussian,
\[v(x)=\pi^{-1/2}\left(B_{\rm av}^U\right)^{1/4}\sigma^{1/2} \chi\big(\sigma^\rho(x-x_\varepsilon)\big) \exp\left(-\frac12\big(B_{\rm av}^U\big)^{1/2} \sigma |x-x_\varepsilon|^2\right)\,,\]
where $\chi\in C_c^\infty\big(\R^2;[0,1]\big)$ is supported in the unit disk and equal to $1$ on $\{|x|\leq \frac12\}$.
 By a change of variable, we see that
 \[\|u\|_{L^2(U)}^2=\|v\|_{L^2(U)}^2=1+o(\sigma^{-1})\]
and
\[ q_\sigma(v,\Ab_{\rm av}^U;U)=B_{\rm av}^U\,\sigma+o(\sigma)\,. \]
Note that $\|u\|_{L^\infty(U)}^2=\pi^{-1}(B_{\rm av}^U)^{1/2}\sigma=\mathcal O(\sigma)$ by \eqref{eq:B-gB-av}. We deduce from Proposition~\ref{prop:app-qf},
\[ \frac{q_\sigma(u,\Ab;U)}{\|u\|_{L^2(U)}^2}\leq (1+\sigma^{-\eta})\sigma B_{\rm av}^U+\mathcal O\big(\|\nabla B\|_{L^2(U)}^2\sigma^{3-4\rho+\eta}\big)\,.\]  
Let us choose $\rho=3/8$ and $\eta=1/8$. Since $U=D(x_\varepsilon,\sigma^{-\rho})$, the error term can be expressed in the following pleasant form
\[ \|\nabla B\|_{L^2(U)}^2\sigma^{3-4\rho+\eta}=\frac1{|U|}\|\nabla B\|_{L^2(U)}^2\pi \sigma^{3-6\rho+\eta}=
\frac1{|U|}\|\nabla B\|_{L^2(U)}^2\pi\sigma^{7/8}\,.\]
So  we infer from \eqref{eq:B-gB-av} that
\[ \]
\[ \frac{q_\sigma(u,\Ab;U)}{\|u\|_{L^2(U)}^2}\leq (1+\sigma^{-\eta})\sigma B(x_\varepsilon) +o(\sigma)\,.\]  
 Since $u$ is supported in $U$, we deduce from the min-max principle \eqref{eq:evD},
\[\lambda(\sigma,\Ab;\Omega)\leq B(x_\varepsilon) \sigma+o(\sigma)\overset{(x_\varepsilon\in M_\varepsilon)}{ \leq} (m_0(B;\Omega)+\varepsilon)\sigma +o(\sigma)\,.\] 
Taking the successive limits, as $\sigma\to+\infty$ then as $\varepsilon\to0_+$, we get
\begin{equation}\label{eq:ev-ub-DM}
\limsup_{\sigma\to+\infty}\frac{\lambda(\sigma,\Ab;\Omega)}{\sigma}\leq m_0(B;\Omega)\,.
\end{equation}
\subsection{Lower bound}

The lower bound in Theorem~\ref{thm:evD} is non-asymptotic and does not require the hypothesis that the essential infimum is strictly positive.

\begin{prop}\label{prop:evD-lb}
Let $\Ab\in H^1(\R^2;\R^2)$ and $B=\curl\Ab$. For all $u\in C_c^\infty(\Omega)$ and $\sigma>0$, the following lower bound holds
\[\int_\Omega|(\nabla-i\sigma\Ab)u|^2\,dx\geq \sigma\int_\Omega B(x)|u(x)|^2\,dx\,.\] 
\end{prop}
\begin{proof}
Consider a sequence $(\Ab_n)_{n\geq 1}\subset C^\infty(\R^2;\R^2)$ such that $\Ab_n\to\Ab$ in $H^1(\R^2;\R^2)$. For all $n\geq 1$, let $B_n=\curl \Ab_n$. Note that $B_n\to B$ in $L^2(\R^2)$.

Fix $u\in C_c^\infty(\Omega)$. Since $\Ab_n$ is smooth, we have (see \cite[Lem.~1.4.1]{FH-b})
\[\int_\Omega|(\nabla-i\sigma\Ab_n)u|^2\,dx\geq \sigma \int_\Omega B_n(x)|u(x)|^2\,dx\,.
\]
It is easy to check that
\begin{multline*}\lim_{n\to+\infty} \int_\Omega|(\nabla-i\sigma\Ab_n)u|^2\,dx= \int_\Omega|(\nabla-i\sigma\Ab)u|^2\,dx\\
\quad{\rm and}\quad \lim_{n\to+\infty}\int_\Omega B_n(x)|u(x)|^2\,dx=\int_\Omega B(x)|u(x)|^2\,dx\,.
\end{multline*}
In fact, 
\[\Big|\|(\nabla-i\sigma\Ab_n)u\|_{L^2(\Omega)} -\|(\nabla-i\sigma\Ab)u\|_{L^2(\Omega)}\Big|\leq 
\sigma\|\Ab_n-\Ab\|_{L^4(\Omega)}\|u\|_{L^4(\Omega)}
\]
and
\[\left|\int_\Omega \big( B_n(x)-B(x)\big)|u(x)|^2\,dx\right|\leq \|B_n-B\|_{L^2(\Omega)}\|u\|_{L^4(\Omega)}^2\,.\]
\end{proof}

\subsection{Proof of Theorem~\ref{thm:evD}}

Collect \eqref{eq:ev-ub-DM} and Proposition~\ref{prop:evD-lb}.

\section{The Ginzburg-Landau model}\label{sec:GL}

This section is devoted to the proof of Theorem~\ref{thm:GL}. Also, in Sec.~\ref{sec:mag-lap}, we use Theorem~\ref{thm:GL} to give a new proof of Theorem~\ref{thm:evD}.

\subsection{Lower bound of GL energy}

In the sequel,  $(\psi,\mA)_{\kappa,H}$ denotes a configuration in the space $H^1(\Omega;\C)\times H^1(\Omega;\R^2)$ such that 
$$\mathcal G(\psi,\mA)=\gse (\kappa,H)\,.$$

Our aim is to prove the following proposition.

\begin{prop}\label{prop:lb-loc}
Given $\epsilon\in(0,1)$, there exist $C,\kappa_0>0$ such that the following inequality holds
$$\mathcal G_0(\psi,\mathcal A;Q_\ell(x_0))\geq   g\big(b B_{\rm av}^\ell(x_0)\big)\kappa^2\ell^2-C\Big(\kappa^{15/8}\ell^2+\kappa^{3/2}\|\nabla B\|_{L^2(Q_\ell(x_0)}^2\Big)\,,$$
where
\begin{itemize}
\item $x_0\in\mathcal J_\ell$\,;
\item $\ell=\kappa^{-3/4}$\,;
\item $(\psi,\mathcal A)_{\kappa,H}$ is a minimizer of the GL functional\,;
\item $H=b\kappa$ and $b\in(\epsilon,\frac1\epsilon)$\,;
\item $\mathcal G_0(\psi,\mathcal A;Q_\ell(x_0))=\displaystyle\int_{Q_\ell(x_0)} \left(|(\nabla-i\kappa H\mathcal A)\psi|^2-\kappa^2|\psi|^2+\frac{\kappa^2}{2}|\psi|^4\right)\,dx$\,.
\end{itemize}
\end{prop}
\begin{proof}
First we notice the  useful inequalities (see \cite[Prop.~4.1\,\&\,Thm.~4.2]{AK})
\begin{equation}\label{eq:apriori}
\|\psi\|_{L^\infty(\Omega)}\leq 1\,,\quad \|(\nabla-i\kappa H\mathcal A)\psi\|_{L^2(\Omega)}\leq |\Omega|\kappa\,,\quad\| \mathcal A-\Fb\|_{C^{0,\alpha}(\overline{\Omega})}\leq \frac{C_\alpha}{\kappa}\,,
\end{equation}
where $\alpha\in(0,1)$ can be chosen in an arbitrary manner.

We set 
\begin{equation}\label{eq:phi-x0}
\phi_{x_0}:=\Big(\mathcal A(x_0)-\Fb(x_0)\Big)\cdot(x- x_0)\quad {\rm and}\quad \mathcal A^{\rm new}=\mathcal A-\nabla\phi_{x_0}\,.
\end{equation}
It is easy to check that
\begin{equation}\label{eq:E=Enew}
\mathcal{G}_0(\psi,\mathcal{A};Q_\ell(x_0))=\mathcal{G}_0(u,\mathcal{A}^{\rm new};Q_\ell(x_0))
\end{equation}
where
\begin{equation}\label{eq:fct-u}
u(x)=e^{-i\kappa H\phi_{x_0}}\psi(x)\,.
\end{equation} 
Writing $\mathcal A^{\rm new}=\Fb+\mathcal A^{\rm new}-\Fb$, we get by Cauchy's inequality, 
$$ |(\nabla-i\kappa H \mathcal A^{\rm new})u|^2\geq 
(1-\kappa^{-1/2})|(\nabla-i\kappa H \Fb)u|^2-\kappa^{1/2}
(\kappa H)^2|(\mathcal A^{\rm new}-\Fb)u|^2\,.
$$
Consequently, we infer from the foregoing inequality and the third inequality in \eqref{eq:apriori}, 
  \begin{multline}\label{eq:en-lb}
 \mathcal{G}_0\big(u,\mathcal{A}^{\rm new};Q_\ell(x_0)\big)\geq(1-\kappa^{-1/2})\mathcal{G}_0\big(u,\Fb,Q_\ell(x_0)\big)
  +\kappa^{1/2} \int_{Q_\ell(x_0)}\left(-\kappa^2|u|^2+\frac{\kappa^2}{2}|u|^4\right)dx\\-Cb^2\ell^{2\alpha}\kappa^{5/2}\int_{Q_\ell(x_0)}|u|^2dx.
\end{multline}
Using that $|u|=|\psi|\leq 1$ by \eqref{eq:apriori}, we can estimate the remainder terms in \eqref{eq:en-lb} as follows
\begin{multline}\label{eq:en-lb-rest}
\kappa^{-1/2}\int_{Q_\ell(x_0)}\left(-\kappa^2|u|^2+\frac{\kappa^2}{2}|u|^4\right)dx-Cb^2\ell^{2\alpha}\kappa^{5/2}\int_{Q_\ell(x_0)}|u|^2dx\\
\geq -\kappa^2\ell^2\big(\kappa^{-1/2}+Cb^2\kappa^{1/2}\ell^{2\alpha}\big)\,.
\end{multline}
In order to estimate the term $\mathcal{G}_0\big(u,\Fb,Q_\ell(x_0)\big)$ in \eqref{eq:en-lb}, we will go from the potential $\Fb$ to the potential $\Ab$ introduced in \eqref{eq:mag-pot}. Let $\vartheta$ be the function in \eqref{eq:A=F} and set
\begin{equation}\label{eq:fct-v}
v=e^{-i\kappa H \vartheta}u\,.
\end{equation}
Then 
\begin{multline}\label{eq:en-lb*}
\mathcal{G}_0\big(u,\Fb;Q_\ell(x_0)\big)=\mathcal{G}_0\big(v,\Ab;Q_\ell(x_0)\big)\\
\geq (1-\kappa^{-1/2})\mathcal{G}_0\Big(w,\Ab_{\rm av}^{Q_\ell(x_0)};Q_\ell(x_0)\Big)-\hat C\kappa^{3/2}\|\nabla B\|_{L^2(Q_\ell(x_0))}^2\\
+\kappa^{-1/2}\int_{Q_\ell(x_0)}\left(-\kappa^2|w|^2+\frac{\kappa^2}2|w|^4\right)\,dx
\end{multline}
where we used Proposition~\ref{prop:app-qf}, with $\sigma=\kappa H=b\kappa^2$ and $\eta=1/4$, to estimate  the $L^2$-norm of $|(\nabla-i\kappa H\Ab)v|$; the function $w$ is expressed in terms of $v$ and the gauge function $\varphi^{Q_\ell(x_0)}$ of Proposition~\ref{prop:app-qf} as follows 
\begin{equation}\label{eq:fct-w}
w(x)=v(x)\exp\left(i\kappa H\varphi^{Q_\ell(x_0)}(x)\right)\,.
\end{equation}
Since $|v|=|u|=|\psi|\leq 1$ by \eqref{eq:apriori}, we infer from \eqref{eq:en-lb*},
\begin{equation}\label{eq:en-lb**}
\mathcal{G}_0\big(u,\Fb;Q_\ell(x_0)\big)\geq (1-\kappa^{-1/2})\mathcal{G}_0\Big(w,\Ab_{\rm av}^{Q_\ell(x_0)};Q_\ell(x_0)\Big)-\hat C\kappa^{3/2}\|\nabla B\|_{L^2(Q_\ell(x_0))}^2-\kappa^{3/2}\ell^2\,.
\end{equation}
Note that $\curl\Ab_{\rm av}^{Q_\ell(x_0)}=B_{\rm av}^\ell(x_0)$ introduced in \eqref{eq:B-av}. We write now a lower bound of the energy $\mathcal{G}_0\Big(w,\Ab_{\rm av}^{Q_\ell(x_0)};Q_\ell(x_0)\Big)$ using the bulk energy function $g(\cdot)$.  To that end, we introduce 
\begin{itemize}
\item $\hat b=\frac{H}{\kappa}B_{\rm av}^\ell(x_0)=bB_{\rm av}^{\ell}(x_0)$\,;
\item $R=\ell\sqrt{\kappa H B_{\rm av}^{\ell}(x_0)}$\,;
\item $h(x)=w\left(\frac{\ell}{R}x+x_0\right)$ for  $x\in Q_R:=(-R/2,R/2)^2$\,;
\item The change of variable $y=\frac{R}{\ell}(x-x_0)$.
\end{itemize} 
It is then easy to check that
$$\mathcal{G}_0\Big(w,\Ab_{\rm av}^{Q_\ell(x_0)};Q_\ell(x_0)\Big)=\dfrac{1}{\hat b}G_{\hat b,Q_R}(h)\geq m(\hat b,R)\geq \frac1{\hat b} \big(g(\hat b)R^2-\tilde CR\big)\,,$$
by \eqref{eq:g(b)}. Inserting the foregoing inequality into \eqref{eq:en-lb**}, then remembering the definition of $\hat b$, choosing $\alpha=\frac56$, and collecting the inequalities in \eqref{eq:en-lb*}, \eqref{eq:en-lb-rest}, \eqref{eq:en-lb},  and \eqref{eq:E=Enew}, we eventually get the following inequality,
\begin{multline*}
\mathcal{G}_0(\psi,\mathcal{A};Q_\ell(x_0))\\\geq
g\left(b B_{\rm av}^{\ell}(x_0)\right)\kappa^2\ell^2
-\check C\left(\kappa^2\ell^2\Big(\kappa^{-1/3}+\kappa^{-1/4}\sqrt{B_{\rm av}^{\ell}(x_0)}\,\Big)-\kappa^{3/2}\|\nabla B\|^2_{L^2(Q_\ell(x_0))}\right)\,.
\end{multline*}
Finally, we apply Lemma~\ref{lem1} with $\zeta=\frac1{16}$.
\end{proof}

\subsection{Upper bound of GL energy}

\begin{prop}\label{prop:ub-loc}
Given $\epsilon\in(0,1)$ and $c_2>c_1>0$, there exist $C,\kappa_0>0$ such that, for all $\kappa\geq \kappa_0$,  the following  holds.

For every $x_0\in\mathcal J_\ell$, with $\ell$ satisfying \eqref{eq:ell}, there exists a function $v_{x_0,\ell}\in H^1_0(Q_\ell(x_0))$ such that 
$$\mathcal G_0(v_{x_0,\ell},\Fb;Q_\ell(x_0))\leq   g\big(b B_{\rm av}^\ell(x_0)\big)\kappa^2\ell^2+C\Big(\kappa^{7/4}\ell^2+\kappa^{3/2}\|\nabla B\|_{L^2(Q_\ell(x_0)}^2\Big)\,,$$
where
\begin{itemize}
\item $\Fb$ is the magnetic potential introduced in \eqref{eq:F}\,;
\item $H=b\kappa$ and $b\in(\epsilon,\frac1\epsilon)$\,;
\item the functional $\mathcal G_0(\cdot,\cdot\,;Q_\ell(x_0))$ is introduced in Proposition~\ref{prop:lb-loc}\,.
\end{itemize}
\end{prop}
\begin{proof}
We choose $b\in(\epsilon,\epsilon^{-1})$ and an arbitrary point $x_0\in\mathcal J_\ell$, with $\ell\approx\kappa^{-3/4}$. We introduce the two parameters (that depend on $x_0$ and $\ell$)
\[\hat b=\frac{H}{\kappa}B_{\rm av}^\ell(x_0)=bB_{\rm av}^{\ell}(x_0)\quad{\rm and}\quad
R=\ell\sqrt{\kappa H B_{\rm av}^{\ell}(x_0)}\,.\]
Let $u_{\hat b,R}\in H^1_0\big((-R/2,R/2)^2\big)$ be a minimizer of the energy functional $m_0(\hat b,R)$.  For all $x\in Q_\ell(x_0)$, we introduce the function $v:=v_{x_0,\ell}\in H^1_0(Q_\ell(x_0))$ as follows
\[ v(x)=\exp\Big(i\kappa H(\varphi+\vartheta)\Big)\,u_{\hat b,R}\left(\frac{R}{\ell}(x-x_0)\right)\,,\]
where $\vartheta$ is the function introduced in \eqref{eq:A=F} and $\varphi:=\varphi^{Q_\ell(x_0)}$ is the function introduced in Proposition~\ref{prop:app-qf}.  
Setting $h=\exp\Big(-i\kappa H(\varphi+\vartheta)\Big)v$, it is easy to check that
\[
\mathcal G_0\left(h,\Ab_{\rm av}^{Q_\ell(x_0)};Q_\ell(x_0)\right)=
\frac1{\hat b} m_0(\hat b,R)\,.
\]
Using \eqref{eq:g(b)}, we get further
\begin{equation}\label{eq:ub-cell}
\mathcal G_0\left(h,\Ab_{\rm av}^{Q_\ell(x_0)};Q_\ell(x_0)\right)\leq g\big(bB_{\rm av}^\ell(x_0)\big)\kappa^2\ell^2+\mathcal O(\kappa\ell)\,.
\end{equation}
Setting $u=\exp\Big(-i\kappa H \vartheta\Big)v$, we get by \eqref{eq:A=F},
\begin{equation}\label{eq:ub-cell*}
\mathcal{G}_0\big(v,{\bf F};Q_{\ell}(x_0)\big)=\mathcal{G}_0\big(u,{\bf A};Q_{\ell}(x_0)\big)\,.
\end{equation}
Now we apply Proposition~\ref{prop:app-qf} with $\sigma=\kappa H=b\kappa^2$, $\rho=3/8$ and $\eta=1/4$; eventually we get
\begin{multline*}
\mathcal{G}_0\big(u,{\bf A};Q_{\ell}(x_0)\big)\leq
(1+\kappa^{-1/2}) \mathcal{G}_0\big(h,\Ab_{\rm av}^{Q_\ell(x_0)};Q_{\ell}(x_0)\big)\\
-\kappa^{-1/2} \int_{Q_{\ell}(x_0)}\left(-\kappa^2|h|^2+\frac{\kappa^2}{2}|h|^4\right)dx+\mathcal O\Big( \kappa^{3}\|\nabla B\|_{L^2(Q_\ell(x_0)}^2\Big)\int_{Q_{\ell}(x_0)}|h|^2\,dx\,.
\end{multline*}
Since $|h|\leq 1$, we get further
\begin{equation}\label{eq:ub-cell**}
\mathcal{G}_0\big(u,{\bf A};Q_{\ell}(x_0)\big)\leq (1+\kappa^{-1/2}) \mathcal{G}_0\big(h,\Ab_{\rm av}^{Q_\ell(x_0)};Q_{\ell}(x_0)\big)+\mathcal O(\kappa^{3/2}\ell^2)+\mathcal O(\kappa^{3/2})\|\nabla B\|_{L^2(Q_\ell(x_0))}^2\,.
\end{equation}
Collecting \eqref{eq:ub-cell**}, \eqref{eq:ub-cell*} and \eqref{eq:ub-cell}, we finish the proof of Proposition~\ref{prop:ub-loc}.
\end{proof}

\subsection{Proof of Theorem~\ref{thm:GL}}\

Now we work under the assumptions of Theorem~\ref{thm:GL}. We fix $\epsilon\in(0,1)$ and assume that $H=b\kappa$ with $b$ varying in $(\epsilon,\epsilon^{-1})$. Recall that $\ell\approx\kappa^{-3/4}$ by \eqref{eq:ell}.\medskip

{\bf Step~1:}

Denote by $(\psi,\mathcal A)_{\kappa,H}$ a minimizing configuration such that $\mathcal G(\psi,\mathcal A)=\gse(\kappa,H)$. Dropping the term $\kappa^2H^2\int_\Omega|\curl(\mathcal A-\Fb)|^2\,dx$ from the energy $\mathcal G(\psi,\mathcal A)$, we get the obvious lower bound
\[
\gse(\kappa,H)=\mathcal{G}(\psi,\mathcal{A})\geq \mathcal{G}_0\big(\psi,\mathcal{A};\Omega\big)= \mathcal{G}_0\big(\psi,\mathcal{A};\Omega_{\ell}\big)+ \mathcal{G}_0\big(\psi,\mathcal{A};\Omega\setminus\Omega_{\ell}\big)
\]
where $\mathcal G_0$ is the energy introduced in Proposition~\ref{prop:lb-loc}, and $\Omega_\ell$ is the domain introduced in \eqref{eq:Om-ell}. Using the uniform bounds $|\psi|\leq 1$ and  $|\Omega\setminus\Omega_\ell|=\mathcal O(\ell)$, we get
\[\mathcal{G}_0\big(\psi,\mathcal{A};\Omega\setminus\Omega_{\ell}\big)\geq -\kappa^2\int_{\Omega\setminus\Omega_\ell}=\mathcal O(\ell\kappa^2)=\mathcal O(\kappa^{5/4})\,.\]
Now, we use the obvious decomposition
$\mathcal{G}_0\big(\psi,\mathcal{A};\Omega_{\ell}\big)=\sum\limits_{x\in\mathcal J_\ell}\mathcal G_0\big(\psi,\mathcal A;Q_\ell(x)\big)$ and apply Proposition~\ref{prop:lb-loc}. Eventually, we get 
\[
\gse(\kappa,H)\geq
\kappa^2\ell^2\sum_{x\in \mathcal{J}_{\ell}}g\big(bB_{\rm{av}}^{\ell}(x_0)\big)-C\kappa^{5/4}-C\sum_{x\in\mathcal{J}_{\ell}}\left(\kappa^{15/8}\ell^2+\kappa^{3/2}\|\nabla B\|^2 _{L^2(Q_{\ell}(x_0))}\right)\,.
\]
Since the squares $(Q_\ell(x))_{x\in\mathcal J_\ell}$ are pairwise disjoint,  $\sum\limits_{x\in\mathcal{J}_{\ell}}\|\nabla B\|^2_{L^2(Q_{\ell}(x_0))}=\|\nabla B\|^2_{L^2(\Omega_\ell)}\leq \|\nabla B\|^2_{L^2(\Omega)}$. Using \eqref{eq:N-ub},  $\sum\limits_{x\in\mathcal{J}_{\ell}}\ell^2=N(\ell)\ell^2=\mathcal O(1)$. Consequently, 
\[ 
\gse(\kappa,H)\geq
\kappa^2\ell^2\sum_{x\in \mathcal{J}_{\ell}}g\big(bB_{\rm{av}}^{\ell}(x_0)\big)+\mathcal O(\kappa^{15/8})\quad(\kappa\to+\infty)\,.
\]

{\bf Step~2:}

We introduce the function $\psi^{\rm trial}\in H^1_0(\Omega)$ as follows
\begin{equation}\label{eq:trial-state}\psi^{\rm trial}(y)=\sum_{x\in\mathcal{J}_{\ell}}{\bf 1}_{Q_{\ell}(x)}v_{x,\ell}(y)\quad (y\in \Omega)\,,
\end{equation}
where, for $x\in \mathcal J_\ell$, $v_{x,\ell}\in H^1_0(Q_\ell(x))$ is the function introduced in Proposition~\ref{prop:ub-loc} and extended by $0$ on $\Omega\setminus Q_\ell(x)$. Clearly, $\gse(\kappa,H)\leq \mathcal{G}\big(\psi^{\rm trial},{\bf F}\big)=\mathcal{G}_0\big(\psi^{\rm trial},{\bf F};\Omega\big)$. Using Proposition~\ref{prop:ub-loc} and that the squares $(Q_\ell(x))_{x\in\mathcal J_\ell}$ are pairwise disjoint, we write
\begin{align*}
\mathcal{G}_0\big(\psi^{\rm trial},{\bf F}\big)&=\sum_{x\in \mathcal{J}_\ell}\mathcal{G}_0\big(v_{x,\ell},{\bf F};Q_\ell(x)\big)
\\
&\leq \kappa^2\ell^2\sum_{x\in\mathcal{J}_{\ell}}g\big(bB_{\text{av}}^{\ell}(x)\big)+C\sum_{x\in\mathcal{J}_{\ell}}\left(\kappa^{7/4}\ell^2+\kappa^{3/2}||\nabla B||^2_{L^2(Q_{\ell}(x))} \right)\\&\leq  \kappa^2\ell^2\sum_{x\in\mathcal{J}_{\ell}}g(bB_{\text{av}}^{\ell}(x))+CN(\ell)\kappa^{7/4}\ell^2+\kappa^{3/2}\|\nabla B\|_{L^2(\Omega)}\\
&=  \kappa^2\ell^2\sum_{x\in\mathcal{J}_{\ell}}g(bB_{\text{av}}^{\ell}(x))+\mathcal O(\kappa^{7/4})\,.
\end{align*}

\subsection{Further remarks}

We collect here some  additional properties for later use. In the sequel,  $(\psi,\mathcal{A})_{\kappa,H}$ denotes a minimizing configuration of the energy in \eqref{eq:GL}. 

We start by a rough estimate of $\mathcal A-\Fb$.  By dropping the positive terms in the inequality $\mathcal G(\psi,\mathcal A)\leq \mathcal G(0,\Fb)=0$ we get the following estimate
\[\|\curl(\mathcal A-\Fb)\|_{L^2(\tilde\Omega)}\leq H^{-1}\|\psi\|_{L^2(\Omega)}\,.\]
Also, $(\psi,\mathcal A)$ being a critical point of the GL energy (see \eqref{eq:Euler-Lag}), we know that $\curl(\mathcal A-\Fb)=0$ on $\partial\tilde\Omega$ (see \cite[Eq.~(10.8b)]{FH-b}); hence, the curl-div inequality \cite[Prop.~D.2.1]{FH-b} yields that $\mathcal A-\Fb\in H^1(\tilde\Omega)$; we deduce then by the Sobolev embedding of $H^1(\tilde\Omega)$ in $L^4(\tilde\Omega)$ that
\begin{equation}\label{eq:A-F-L4}
\|\mathcal A-\Fb\|_{L^4(\tilde\Omega)}\leq C_*\|\curl(\mathcal A-\Fb)\|_{L^2(\tilde\Omega)}\leq C_*H^{-1}\|\psi\|_{L^2(\Omega)}\,,
\end{equation}
where $C_*$ depends on $\tilde\Omega$.

We mention some additional properties that follow along the proof of Theorem~\ref{thm:GL} ( see e.g. \cite[Thm~1.2~\&~p.~6636]{AK}). Firstly, we have the improved estimate for the magnetic energy
\[\|\curl(\mathcal A-\Fb)\|_{L^2(\tilde\Omega)}=\mathcal O(\kappa^{-1/8})\,,\]
and also for the energy of $\psi$,
\[\mathcal G_0(\psi,\mathcal A):=\int_\Omega\Big(|(\nabla-i\kappa H\mathcal A)\psi|^2-\kappa^2|\psi|^2+\frac{\kappa^2}2|\psi|^4\Big)\,dx=  \gse^{\rm asy}(b,\ell)\kappa^2+o(\kappa^2)\,.\]
We infer  from  \eqref{eq:Euler-Lag} that $\mathcal G_0(\psi,\mathcal A)=-\frac{\kappa^2}{2}\|\psi\|_{L^4(\Omega)}^4$, which eventually yields the following formula for
 the $L^4$-energy of the order parameter,
\begin{equation}\label{eq:psi-L4}
 \|\psi\|_{L^4(\Omega)}^4\leq  -2\gse^{\rm asy}(b,\ell) +\mathcal O(\kappa^{-1/8})\,.
 \end{equation}

\Bk

\subsection{Application: The Dirichlet Laplacian}\label{sec:mag-lap}

Assuming the hypothesis in Theorem~\ref{thm:evD} on the domain $\Omega$, we will derive an asymptotic upper bound on the eigenvalue $\lambda(\sigma,\Ab;\Omega)$, by constructing a trial state related to the GL order parameter. 

Under the hypothesis in Theorem~\ref{thm:evD}, it is sufficient to handle the case where the domain $\Omega$ consists of a single connected component. In fact, by the min-max principle, $\lambda(\sigma,\Ab;\Omega)=\min\limits_{1\leq i\leq N} \lambda(\sigma,\Ab;\Omega_i)$. 

In the sequel, we assume that $\Omega$ is  connected and its boundary consists of a finite number of connected components (as in Sec.~\ref{sec:GL}). Recall
 the divergence free magnetic potential, $\Fb$, introduced in \eqref{eq:F}. In light of the relation \eqref{eq:A=F}, we observe that
\begin{equation}\label{eq:mu1}
\lambda(\sigma,\Ab;\Omega)=\lambda(\sigma,\Fb):=\inf_{u\in H^1(\Omega)\setminus\{0\}}
\frac{\|(\nabla-i\sigma\Fb)u\|^2_{L^2(\Omega)}}{\|u\|^2_{L^2(\Omega)}}\,.
\end{equation}

The hypothesis $m_0(B;\Omega)>0$ yields that $B(x)\geq c>0$ a.e. on $\Omega$, where $c=m_0(B;\Omega)>0$ is  constant.
This allows us to benefit from the results and the analysis of Sec.~\ref{sec:GL}.

\subsubsection{Link with the GL energy}

In the sequel,  we set
\begin{equation}\label{eq:ell'}
\ell=\ell_\sigma:= \sigma^{-3/8}\,.
\end{equation}
We fix $a\in(0,1)$ and introduce the parameters
\begin{equation}\label{eq:kap-sigm-b}
 b=\dfrac{1-a}{m_0(B;\Omega)}\,,\quad \kappa= b^{-1/2}\sigma^{1/2}\quad{\rm and}\quad H=b\kappa\,. 
\end{equation}
The conditions in  \eqref{eq:ell'} and \eqref{eq:kap-sigm-b} ensure that, as $\sigma\to+\infty$, the configuration $(\kappa,H,\ell, b)$ satisfies the requirements for using Theorem~\ref{thm:GL}. 
In particular,
\begin{equation}\label{eq:GLenD-asy}
\gse (\kappa,H)=\kappa^2\gse^{\rm asy}(b,\ell)+\mathcal O(\kappa^{15/8})\,.
\end{equation}
Furthermore, by Remark~\ref{rem:GL} and \eqref{eq:l-o-t=0}, there exist  constants $c_a'>c_a>0$ such that
\begin{equation}\label{eq:enD-asym<0}
-c_a\leq \gse^{\rm asy}(b,\ell)\leq -c_a'
\end{equation}
and
\begin{equation}\label{eq:ca}
c_a,c_a'=\mathcal O(a)\quad(a\to0_+)\,.
\end{equation}
Next we pick a minimizing configuration $(\psi,\mathcal A)_{\kappa,H}$. Collecting \eqref{eq:enD-asym<0} and \eqref{eq:psi-L4}, we obtain
\begin{equation}\label{eq:psi-L4*}
2c_a\leq  \|\psi\|_{L^4(\Omega)}^4\leq  2c_a'\,.
 \end{equation}
Consequently, since $|\psi|\leq 1$ everywhere, we get
\begin{equation}\label{eq:psi>0}
\int_\Omega|\psi|^2\,dx\geq \int_\Omega|\psi|^4\,dx\geq 2c_a>0\,.
\end{equation}  \Bk
\begin{rem}\label{rem:bB>1a.e.}
Fix an arbitrary positive number $\varepsilon<1$. Let us assume that  $\mathrm E^{\rm asy}(b,\ell)\underset{\ell\to0_+}{=}o(1)$ and set $n_\ell={\rm Card}\mathcal J'_\ell$ where $\mathcal J'_\ell:=\{x\in\mathcal J_\ell,~bB_{\rm av}^\ell(x)\leq 1-\varepsilon\}$ and $\mathcal J_\ell$ is introduced in \eqref{eq:J}.  Since $g\leq 0$ and monotone increasing,
\[ o(1)=\mathrm E^{\rm asy}(b,\ell):=\ell^2\sum_{x\in\mathcal J_\ell} g(bB_{\rm av}^\ell(x))\leq \ell^2\sum_{x\in\mathcal J_\ell'} g(bB_{\rm av}^\ell(x)) \leq \ell^2 n_\ell g(1-\varepsilon )\leq0,\]
hence, $n_\ell=o(\ell^{-2})$. This yields, by  Lebesgue's differentiation theorem, that $bB(z)\geq 1-\varepsilon$ a.e. in $\Omega$.
In fact, if we pick $z\in\Omega$ and $\delta_1(z)>0$ so that $D(z,\delta_1(z))\subset\Omega$, then  for any fixed  $\delta\in(0,\delta_1(z))$, we have $\int_{D(z,\delta)} bB(y)dy\geq (1-\varepsilon)|D(z,\delta)| $, since
\[
\int_{D(z,\delta)} bB(y)dy\geq 
\sum\limits_{\substack{x\in\mathcal J_\ell\setminus \mathcal J_\ell'\\ Q_\ell(x) \subset D(z,\delta)}} \int_{Q_\ell(x)} b B(y)dy\geq (1-\varepsilon)\ell^2\mathcal N_\ell\]
where $\mathcal N_\ell= {\rm Card}(\{x\in\mathcal J_\ell\setminus \mathcal J_\ell',~ Q_\ell(x)\subset D(z,\delta)\})=|D(z,\delta)|\ell^{-2}+o(\ell^{-2})$ as $\ell\to0$\,.
\end{rem}
\subsubsection{The trial state}
We introduce a cut-off function $\chi_\ell\in C_c^\infty(\Omega)$ in order to produce a trial state in $H^1_0(\Omega)$. We choose $\chi_\ell$ such that
\[\chi_\ell(x)=1{\rm ~for~} {\rm dist}(x,\partial\Omega)>2\ell\,,\quad {\rm and}~0\leq \chi_\ell\leq 1, ~|\nabla\chi_\ell |\leq C_0\ell^{-1}~{\rm in~}\Omega\,.\]
Using \eqref{eq:Euler-Lag}, we check that
\begin{equation}\label{eq:ub-trial-state-m}
\begin{aligned}
\|(\nabla-i\kappa H\mathcal A)(\chi_\ell\psi)\|_{L^2(\Omega)}^2&={\rm Re}\langle -(\nabla-i\kappa H\mathcal A)\psi,\chi_\ell^2\psi\rangle_{L^2(\Omega)}+\| \psi \nabla\chi_\ell\|_{L^2(\Omega)}^2\\
&\leq \kappa^2 \|\chi_\ell\psi\|_{L^2(\Omega)}^2+C_0^2\ell^{-2}\|\psi\|_{L^2(\Omega)}^2\,.
\end{aligned}
\end{equation}
By the simple identity $\mathcal A=\Fb+(\mathcal A-\Fb)$ and  Cauchy's inequality,  we write, for any $\delta\in(0,1)$,
\begin{equation}\label{eq:lb-via-mu1}
\|(\nabla-i\kappa H\mathcal A)(\chi_\ell\psi)\|_{L^2(\Omega)}^2\geq (1-\delta)\|(\nabla-i\kappa H\mathcal \Fb)(\chi_\ell\psi)\|_{L^2(\Omega)}^2-\delta^{-1}\kappa^2H^2\|(\mathcal{A}-{\bf F})\psi\|^2_{L^2(\Omega)}.
\end{equation}
We estimate the term $\|(\mathcal{A}-{\bf F})\psi\|^2_{L^2(\Omega)}$ by using H\"older's inequality, and the two estimates in \eqref{eq:A-F-L4} and \eqref{eq:psi-L4*}; eventually, we get
\[
\|(\mathcal{A}-{\bf F})\psi\|^2_{L^2(\Omega)}\leq \|\mathcal A-\Fb\|_{L^4(\Omega)}^2\|\psi\|_{L^4(\Omega)}^2\leq \dfrac{C}{H^2}(2c_a')^{1/2}\int_{\Omega}|\psi|^2dx\,.\]
We insert this into \eqref{eq:lb-via-mu1} to get (note that $(1-\delta)^{-1}\leq 2$)
\[
\|(\nabla-i\kappa H\mathcal A)(\chi_\ell\psi)\|_{L^2(\Omega)}^2\geq (1-\delta)\Big(\|(\nabla-i\kappa H\mathcal \Fb)(\chi_\ell\psi)\|_{L^2(\Omega)}^2
-2\kappa^2\delta^{-1}C(2c_a')^{1/2}\int_{\Omega}|\psi|^2dx\Big)\,.
\]
Now we infer from \eqref{eq:ub-trial-state-m},
\[\|(\nabla-i\kappa H\mathcal \Fb)(\chi_\ell\psi)\|_{L^2(\Omega)}^2\leq \frac1{1-\delta} \Big(\kappa^2 \|\chi_\ell\psi\|_{L^2(\Omega)}^2+C_0^2\ell^{-2}\|\psi\|_{L^2(\Omega)}^2 \Big)+2\kappa^2\delta^{-1}C(2c_a')^{1/2}\int_{\Omega}|\psi|^2dx\,.\]
By \eqref{eq:kap-sigm-b}, $\kappa H=\sigma$. Then, in light of \eqref{eq:mu1}, we deduce that
\begin{equation}\label{eq:ub-evD-mod*}
\lambda(\sigma,\Fb)\leq  \frac{\kappa^2}{1-\delta} +\Big(\frac{C_0^2}{1-\delta}\ell^{-2} +2\kappa^2\delta^{-1}C(2c_a')^{1/2}\Big) \frac{\|\psi\|^2_{L^2{(\Omega})}}{\|\chi_\ell\psi\|^2_{L^2{(\Omega})}}\,.
\end{equation}
Since $\chi_\ell=1$ on $\{{\rm dist}(x,\partial\Omega)>2\ell\}$, we get from \eqref{eq:psi>0} a constant $M_a>0$ such that,
\[\|\chi_\ell\psi\|^2_{L^2{(\Omega})}\geq (1-M_a\ell)  \|\psi\|^2_{L^2{(\Omega})}\geq \frac12\|\psi\|_{L^2(\Omega)}^2 ~{\rm for~}\ell~{\rm close~to~}0\,. \]
Furthermore, by \eqref{eq:kap-sigm-b}, $\kappa^2=(1-a)^{-1}m_0(B;\Omega)\sigma$. And by \eqref{eq:ell'}, $\ell=\sigma^{-3/8}$. Therefore, we deduce from \eqref{eq:ub-evD-mod*},
\[\lambda(\sigma,\Fb)\leq \frac{m_0(B;\Omega)}{(1-\delta)(1-a)} \sigma +\frac{2C_0^2}{1-\delta}\sigma^{3/4}+ \frac{4C\delta^{-1}(2c_a')^{1/2} m_0(B;\Omega)}{1-a}\sigma\,.\]
Taking the successive limits, $\sigma\to+\infty$, $a\to0_+$ and $\delta\to0_+$, we
get
\[\limsup_{\sigma\to+\infty}\Big(\sigma^{-1}\lambda(\sigma,\Fb)\Big)\leq m_0(B;\Omega)\,.\]
Note that we make use of \eqref{eq:ca} which ensures that $c_a$ vanishes as $a$ approaches $0$.  
\Bk

\section*{Acknowledgments} 
The authors would like to thank B. Helffer and N. Raymond for the valuable comments on a preliminary version of this paper, and the anonymous referees for pointing an error in the formulation of Prop.~3.1 and a suggestion to  simplify the proof of Thm.~1.2.  A.K. is supported by the Lebanese University in the framework of the project ``Analytical and Numerical Aspects of the Ginzburg-Landau Model''. Part of this work has been carried out at  CAMS (\emph{Center for Advanced Mathematical Sciences}, Beirut). The authors acknowledge its hospitality.

\end{document}